\documentclass[a4paper, 12pt]{article}%

\usepackage{amsmath}
\usepackage{amsfonts}
\usepackage{amssymb}
\usepackage{amsthm}
\usepackage{graphicx}
\usepackage{titlesec}
\usepackage{tikz}
\usepackage{algorithm2e}

\newtheorem{theorem}{Theorem}
\newtheorem{lemma}[theorem]{Lemma}
\newtheorem{definition}[theorem]{Definition}
\newtheorem{remark}[theorem]{Remark}
\newtheorem{proposition}[theorem]{Proposition}

\newcommand{\tri}{\mathcal{T}}
\newcommand{\Tau}{\mathcal{T}}
\newcommand{\triH}{\tri_H}

\newcommand{\dx}{\operatorname*{d}\hspace{-0.3ex}x}

\newcommand{\diam}{\operatorname*{diam}}
\newcommand{\support}{\operatorname*{supp}}

\newcommand{\Inodal}{\mathfrak{I}_{H}}

\newcommand{\umshk}{u^{\operatorname*{ms},k}_{H,h}}
\newcommand{\umshkn}{u^{\operatorname*{ms},k,(n)}_{H,h}}
\newcommand{\vmshk}{v^{\operatorname*{ms},k}_{H,h}}
\newcommand{\wmshk}{w^{\operatorname*{ms},k}_{H,h}}
\newcommand{\umsh}{u^{\operatorname*{ms}}_{H,h}}
\newcommand{\tildeumsh}{\tilde{u}^{\operatorname*{ms}}_{H,h}}
\newcommand{\vms}{v^{\operatorname*{ms}}_H}

\newcommand{\ums}{u^{\operatorname*{ms}}}
\newcommand{\umsn}{u^{\operatorname*{ms},(n)}}

\newcommand{\wfine}{w^{\operatorname*{f}}}

\newcommand{\Cint}{C_{\Inodal}}

\newcommand{\Vms}{V^{\operatorname*{ms}}}

\newcommand{\VmsHh}{V^{\operatorname*{ms}}_{H,h}}
\newcommand{\VmsHk}{V^{\operatorname*{ms}}_{H,k}}
\newcommand{\VmsHkh}{V^{\operatorname*{ms},k}_{H,h}}
\newcommand{\Vf}{V^{\operatorname*{f}}}

\newcommand{\R}{\mathbb{R}}

\newtheorem{assumption}{Assumption}

\begin{document}


\begin{center}
{\LARGE A localized orthogonal decomposition method\\[0.5em] for semi-linear elliptic problems}\\[2em]
\end{center}

\renewcommand{\thefootnote}{\fnsymbol{footnote}}
\renewcommand{\thefootnote}{\arabic{footnote}}

\begin{center}
{\large Patrick Henning\footnote[1]{Department of Information Technology, Uppsala University, Box 337, SE-751 05 Uppsala, Sweden}$\hspace{0pt}^{\ast}$,
 Axel M\r{a}lqvist$\hspace{0pt}^{\hspace{1pt}1}$\renewcommand{\thefootnote}{\fnsymbol{footnote}}\setcounter{footnote}{0}
 \hspace{-3pt}\footnote{A. M\r{a}lqvist and P. Henning are supported by The G\"{o}ran Gustafsson Foundation and The Swedish Research Council.},
 \renewcommand{\thefootnote}{\arabic{footnote}}\setcounter{footnote}{2}
 Daniel Peterseim\footnote[2]{Institut f\"ur Mathematik, Humboldt-Universit\"at zu Berlin, Unter den Linden 6, D-10099 Berlin, Germany}\renewcommand{\thefootnote}{\fnsymbol{footnote}}\setcounter{footnote}{1}
 \hspace{-3pt}\footnote{D. Peterseim is supported by the DFG Research Center Matheon Berlin through project C33.}}\\[2em]
\end{center}

\begin{center}
{\large{\today}}
\end{center}

\begin{center}
\end{center}

\begin{abstract}
In this paper we propose and analyze a localized orthogonal decomposition (LOD) method for solving semi-linear elliptic problems with heterogeneous and highly variable coefficient functions. For this purpose we construct a generalized finite element basis that spans a low dimensional multiscale space. The basis is assembled by performing localized linear fine-scale computations on small patches that have a diameter of order $H\log(H^{-1})$ where $H$ is the coarse mesh size. Without any assumptions on the type of the oscillations in the coefficients, we give a rigorous proof for a linear convergence of the $H^1$-error with respect to the coarse mesh size. To solve the arising equations, we propose an algorithm that is based on a damped Newton scheme in the multiscale space.
\end{abstract}



\section{Introduction}

This paper is devoted to the numerical approximation of solutions of semi-linear elliptic problems with rapidly oscillating and highly varying coefficient functions. We are concerned with the following type of equations:
\begin{align*}
- \nabla \cdot (A\nabla u) +F(u,\nabla u) = g.
\end{align*}
Here, we prescribe a (zero-) Dirichlet boundary condition for $u$, $A$ is a highly variable diffusion matrix and $F$ is a highly variable nonlinear term that typically describes advective and reactive processes. In particular, we have a linear term of second order and nonlinear terms of order $1$ and $0$.
A typical application is the stationary (Kirchhoff transformed) Richards equation that describes the groundwater flow in unsaturated soils (c.f. \cite{Alt:Luckhaus:1983,Berninger:2009,Berninger:Kornhuber:Sander:2007}). The equation reads
\begin{align*}
\nabla \cdot ( K \nabla u ) - \nabla \cdot ( K \hspace{2pt} kr(M(u)) \vec{g}) = f,
\end{align*}
where $u$ is the so-called generalized pressure, $K$ is the hydraulic conductivity in the soil, $kr$ the relative permeability depending on the saturation, $M$ is some nonlinearity arising from the Kirchhoff transformation and $\vec{g}$ denotes the gravity vector. If we add an infiltration process, the equation receives an additional nonlinear reaction term.

The numerical treatment of such equations is often complicated and expensive. Due to the high variability of the coefficient functions, one requires extremely fine computational grids that are able to capture all the fine scale oscillations. Using standard methods such as Finite Element or Finite Volume schemes, this results in systems of equations of enormous size and therefore in a tremendous computational demand that can not be handled in a lot of scenarios. 

There is a large variety of so-called multiscale methods that are able to overcome this difficulty by decoupling the fine scale computations into local parts. This decreases the computational demand without suffering from a remarkable loss in accuracy. Examples of multiscale methods are the Heterogeneous Multiscale Method (HMM) by E and Engquist \cite{E:Engquist:2003} and the Multiscale Finite Element Method (MsFEM) proposed by Hou and Wu \cite{Hou:Wu:1997}. Both methods fit into a common framework and are strongly related to numerical homogenization (c.f. \cite{Gloria:2006,Henning:Ohlberger:2011_2,Henning:2012}). HMM and MsFEM are typically not constructed for a direct approximation of exact solutions but for homogenized solutions and corresponding correctors instead. This implies that they are only able to approximate the exact solution up to a modeling error that depends on the homogenization setting (c.f. \cite{Gloria:2006}). General proofs of convergence are therefore hard to achieve. 

We are concerned with a multiscale method that is based on the concept of the Variational Multiscale Method (VMM) proposed by Hughes et al. \cite{Hughes:et-al:1998}. In comparison to HMM and MsFEM, the VMM aims to a direct approximation of the exact solution without suffering from a modeling error remainder arising from homogenization theory. The key idea of the Variational Multiscale Method is to construct a splitting of the original solution space $V$ into the direct sum of a low dimensional space for coarse grid approximations and high dimensional space for fine scale reconstructions. In this work, we consider a modification and extension of this idea that was developed in \cite{LaMa07,Malqvist:2011} and that was explicitly proposed in \cite{MaPe12}. Here, the splitting is such that we obtain an accurate but low dimensional space $V^{\operatorname*{ms}}$ (where we are looking for our fine scale approximation instead of an approximation of a coarse part) and a high dimensional residual space $V^{\operatorname*{f}}$. The construction of $V^{\operatorname*{ms}}$ involves the computation of one fine scale problem in a small patch per degree of freedom. Mesh adaptive versions of the VMM with patch size control that can be also applied to this multiscale method were achieved in \cite{LaMa07,Larson:Malqvist:2009,Larson:Malqvist:2009_2,Nordbotten:2008}. The first rigorous proof of convergence was recently obtained in \cite{MaPe12} for linear diffusion problems.

In this contribution, we present an efficient way of handling semi-linear elliptic multiscale problems in the modified VMM framework, including a proof of convergence based on the techniques established in \cite{MaPe12}. Even though the original problem is nonlinear, the fine scale computations are purely linear problems that can be solved in parallel. The main result of this article is the optimal convergence of the $H^1$-error between exact solution $u$ and its multiscale approximation $u^{\operatorname{ms}}_H$. We show that, if the patch size is of order $H\log(H^{-1})$, the following error bound holds true:
\begin{equation*}
 \| u-u^{\operatorname{ms}}_H \|_{H^1(\Omega)} \leq C H.
\end{equation*}
Here, $C$ denotes a generic constant independent of the mesh size of the computational grid or the oscillations of $A$ and $F$.

$\\$
The paper is structured as follows. In Section \ref{s:setting} we introduce the setting of this paper, including the assumptions on the considered semi-linear problem. In Section \ref{s:locbasis} we present and motivate our method and we state the corresponding optimal convergence result. The result itself is proved in Section \ref{s:error}. Finally, to solve the arising nonlinear equations, we propose an algorithm that is based on a damped Newton scheme in the multiscale space. This is done in Section \ref{s:newtonms}. In Section \ref{s:numericalexperiments} we emphasize our theoretical results by a numerical experiment.

\section{Setting}
\label{s:setting}

Let $\Omega\subset\mathbb{R}^{d}$ be a bounded Lipschitz domain with polyhedral boundary, let $V:=H_{0}^{1}( \Omega )$ and let $A\in L^{\infty}(\Omega,\mathbb{R}^{d \times d}_{\mathrm{sym}})$ denote a matrix valued function with uniformly strictly positive eigenvalues. We assume that the space $H_{0}^{1}( \Omega )$ is endowed with the $H^1$-semi norm given by $|v|_{H^1(\Omega)}:=\| \nabla v\|_{L^2(\Omega)}$ (which is equivalent to the common $H^1$-norm in $H^1_0(\Omega)$). By $ \langle \cdot, \cdot  \rangle := (\cdot,\cdot)_{L^2(\Omega)}$ we denote the inner product in $L^2(\Omega)$ and $F: \Omega \times \R \times \R^d \rightarrow \R$ is a nonlinear measurable function.

$\\$
For a given source term $g \in L^2(\Omega)\subset H^{-1}(\Omega)$ we are concerned with the problem to find $u\in H^1_0(\Omega)$ (i.e. with a homogeneous Dirichlet boundary condition) with
\begin{align}
\label{originalproblem}
\langle A\nabla u,\nabla v\rangle+\langle F(\cdot,u,\nabla u),v\rangle =\langle g, v\rangle
\end{align}
for all test functions $v\in H^1_0(\Omega)$. To simplify the notation, we define the operator $B : H^1_0(\Omega) \rightarrow H^{-1}(\Omega)$ by
\begin{align*}
\langle B(v), w \rangle_{H^{-1},H^1_0} := \langle A\nabla v,\nabla w\rangle+\langle F(\cdot,v,\nabla v),w\rangle \quad \mbox{for} \enspace v,w \in H^1_0(\Omega),
\end{align*}
where $\langle \cdot, \cdot \rangle_{H^{-1},H^1_0}$ denote the dual pairing in $H^1_0(\Omega)$. Here, $A$ is a rapidly oscillating, highly heterogeneous diffusion matrix. $F(\cdot,\xi,\zeta)$ is also allowed to be rapidly oscillating, without further assumptions on the type of the oscillations. 

We assume that $A$ and $F$ are of the same order of magnitude. 
This prohibits for instance advection dominated processes, which must be treated with a different approach and requires a different set of multiscale basis functions than the one proposed in this paper. In particular, we will show that if the $F$-term does not dominate the equation, it is sufficient to construct a multiscale spaced only based on the oscillations of $A$. This implies that we only need to solve linear problems on the fine scale.
If $A$ and $F$ are not of the same size, the proposed method needs modifications with respect to the construction of the multiscale basis. Typical examples where the lower term is dominant are models for transport of solutes in groundwater where we deal with extremely large P\'eclet numbers and a corresponding scaling of the advective terms. In this case, the oscillations of $F$ become significant to obtain accurate upscaled and homogenized properties (c.f. \cite{Henning:Ohlberger:2010,Henning:Ohlberger:2011}).

$\\$
For the subsequent analytical considerations and in order to guarantee a unique solution of (\ref{originalproblem}), we make the following assumptions:
\begin{assumption}
We assume:
\begin{itemize}
\item[(A1)] $A\in L^\infty\left(\Omega,\mathbb{R}_{\mathrm{sym}}^{d\times d}\right)$ with 
\begin{equation*}%
\infty>\beta :=\|A\|_{L^{\infty}(\Omega)}=\underset{x\in\Omega}{\operatorname{ess}\sup}%
\sup\limits_{\zeta\in\mathbb{R}^{d}\setminus\{0\}}\dfrac{A( x) \zeta \cdot \zeta
}{|\zeta|^2}.
\end{equation*}
and there exists $\alpha$ such that
\begin{equation*}%
0<\alpha :=\underset{x\in\Omega}{\operatorname{ess}\inf}%
\inf\limits_{\zeta\in\mathbb{R}^{d}\setminus\{0\}}\dfrac{ A( x) \zeta \cdot \zeta
}{ |\zeta|^2},
\end{equation*}
\item[(A2)] There exist $L_1,L_2 \in \R_{>0}$ such that uniformly for almost every $x$ in $\Omega$:
\begin{align*}
 |F(x,\xi_1,\zeta) - F(x,\xi_2,\zeta)| &\le L_1 |\xi_1 - \xi_2|, \quad \mbox{for all } \zeta \in \R^d, \enspace \xi_1,\xi_2 \in \R,\\
 |F(x,\xi,\zeta_1) - F(x,\xi,\zeta_2)| &\le L_2 |\zeta_1 - \zeta_2|, \quad \mbox{for all } \zeta_1, \zeta_2 \in \R^d, \enspace \xi \in \R,\\
  F(x,0,0) &= 0.
\end{align*}
\item[(A3)] 
$B$ is strongly monotone, i.e. there exist $c_0>0$ so that for all $u,v \in H^1_0(\Omega)$:
\begin{align}
\label{strong-monotonicity}\langle B(u) - B(v), u-v \rangle_{H^{-1},H^1_0} &\ge c_0 | u - v|_{H^1(\Omega)}^2.
\end{align}
\end{itemize}
\end{assumption}
Under assumptions (A1)-(A3), the Browder-Minty theorem (c.f. \cite{Ruzicka:Book:2004}, Section 3, Theorem 1.5 therewithin) yields a unique solution of problem (\ref{originalproblem}).

Typically, the validity of assumption (A3) can be checked by looking at the properties of the nonlinear function $F$. For instance, if there exists a constant $\alpha_0 \ge 0$, such that $\partial_{\xi}F(x,\xi,\zeta) \ge \alpha_0$ for all $\zeta$ and almost every $x$ (i.e. $F(x,\cdot,\zeta)$ is monotonically increasing) and if $\alpha_0$ and $L_2$ are such that $L_2\le 2 \alpha_0$ and $L_2 < 2 \alpha$ then (A3) is fulfilled. This can be checked by a simple calculation:
\begin{eqnarray*}
\lefteqn{\langle B(u) - B(v), u-v \rangle_{H^{-1},H^1_0}}\\
&\ge& \alpha ||  \nabla u -  \nabla v ||_{L^2(\Omega)}^2 + \alpha_0 ||  u - v ||_{L^2(\Omega)}^2 - L_2 (|  u - v |, |  \nabla u -  \nabla v |)_{L^2(\Omega)} \\
&\ge& (\alpha - \frac{L_2}{2}) ||  \nabla u -  \nabla v ||_{L^2(\Omega)}^2 + (\alpha_0- \frac{L_2}{2}) ||  u - v ||_{L^2(\Omega)}^2. 
\end{eqnarray*}
\begin{remark}
Let the $C_p$ denote the constant appearing in the Poincar\'e-Friedrichs inequality for $H^1_0(\Omega)$ functions. Observe that (A1)-(A3) imply that the solution $u\in H^1_0(\Omega)$ of (\ref{originalproblem}) fulfills
\begin{eqnarray}
\label{bound_F_u_grad_u} \nonumber\lefteqn{\| F(u,\nabla u)\|_{L^2(\Omega)} \le \| F(u,\nabla u) - F(0,\nabla u)\|_{L^2(\Omega)} + \| F(0,\nabla u) - F(0,0)\|_{L^2(\Omega)}}\\
&\le& (L_1 C_p + L_2) |u|_{H^1(\Omega)} \le C_p\frac{L_1 C_p + L_2}{c_0 } \|g\|_{L^2(\Omega)}.\hspace{130pt}
\end{eqnarray}
\end{remark}
Note that problem (\ref{originalproblem}) also covers equations such as
\begin{equation*}
-\nabla\cdot\left( \kappa(u)A\nabla u \right)=F(u,\nabla u),
\end{equation*}
for a strictly positive and sufficiently regular function $\kappa$ (independent of $x$). In this case, the equation can be rewritten as
\begin{equation*}
-\nabla\cdot A\nabla u=\tilde{F}(u,\nabla u).
\end{equation*}

In the remainder of this paper, we use the notation $q_1 \lesssim q_2$ if $q_1 \le C q_2$ where $C>0$ is a constant that only depends on the shape regularity of the mesh, but not on the mesh size. Dependencies such as $(L_1+L_2)\alpha^{-1}$ are always explicitly stated. Dependencies on the contrast $\frac{\beta}{\alpha}$ are allowed to be contained in $\lesssim$.
\section{A Rigorous Multiscale Method}\label{s:locbasis}

In this section we suggest a local orthogonal decomposition multiscale method (LOD) that is based on the concept introduced by Hughes et al. \cite{Hughes:et-al:1998} and the specific constructions proposed in \cite{LaMa07,Malqvist:2011} for linear problems. The required multiscale (MS) basis functions are obtained with the strategy established in \cite{MaPe12}.

The main idea of the Variational Multiscale Method is to start from a finite element space $\mathcal{V}_h$ with a highly resolved computational grid and to construct a splitting of this space into the direct sum $\mathcal{V}_h=\mathcal{V}^l \oplus \mathcal{V}^{\operatorname*{f}}$ of a low dimensional space $\mathcal{V}^{\operatorname*{l}}$ and a 'detail space' $\mathcal{V}^{\operatorname*{f}}$ containing all the missing oscillations. Then, a basis of $\mathcal{V}^l$ is assembled and we can compute a Galerkin approximation $u_l$ of $u$ in $\mathcal{V}^{\operatorname*{l}}$. However, the success of this approach strongly depends on the choice of $\mathcal{V}^l$. On the one hand, the costs for assembling a basis of $\mathcal{V}^{\operatorname*{l}}$ must be kept low. On the other hand, the basis functions somehow need to contain information about fine scale features. For instance, a standard coarse finite element space is cheap to assemble but will fail to yield reliable approximations. On the contrary, the space spanned by high resolution finite element approximations yields perfect approximations,
but is as costly as the original problem that we tried to avoid. Therefore, the key is to find an optimal balance between costs and accuracy. In previous works (c.f. \cite{Hughes:et-al:1998,LaMa07,Larson:Malqvist:2009}) the multiscale basis (MS-basis) of $\mathcal{V}^l$ was constructed involving the full multiscale operator $B$ that corresponds with the left hand side of the original problem. In a fully linear setting, this can be a reasonable choice. However, it gets extremely expensive if $B$ is a nonlinear operator, since it leads to numerous nonlinear equations to solve. Furthermore it is not clear if the constructed set of basis functions leads to good approximations. One novelty of this work is that we do not involve the full operator $B$ in the construction of the MS-basis, but only the linear diffusive part  $\langle A \nabla \cdot, \nabla \cdot \rangle$. Even though the oscillations of $F$ are not captured by the MS-basis, we can show that we are still able to obtain accurate approximations and to preserve the optimal convergence rates.

\subsection{Notation and discretization}\label{ss:classical}
Let $\triH$ denote a regular triangulation of $\Omega$ and let $H:\overline\Omega\rightarrow\mathbb{R}_{>0}$ denote the $\triH$-piecewise constant mesh size function with $H\vert_T=H_T:=\diam(T)$ for all $T\in\triH$. Additionally, let $\mathcal{T}_h$ be a regular triangulation of $\Omega$ that is supposed to be a refinement of $\mathcal{T}_H$. We assume that $\mathcal{T}_h$ is sufficiently small so that all fine scale features of $B$ are captured by the mesh.  The mesh size $h$ denotes the maximum diameter of an element of $\mathcal{T}_h$.
The corresponding classical (conforming) finite element spaces of piecewise polynomials of degree $1$ are given by
\begin{align*}
V_H&:=\{v_H\in C^0(\bar{\Omega})\cap H^1_0(\Omega)\;\vert\;\forall T\in\triH:\enspace (v_H)\vert_T \text{ is affine}\},\\
V_h&:=\{v_h\in C^0(\bar{\Omega})\cap H^1_0(\Omega)\;\vert\;\forall K\in\mathcal{T}_h: \enspace (v_h)\vert_K \text{ is affine}\}.
\end{align*}
By $J$, we denote the dimension of $V_H$ and by $\mathcal{N}_H=\{ z_j| \hspace{2pt} 1 \le j \le J\}$ the set of interior vertices of $\triH$. For every vertex $z_j \in\mathcal{N}_H$, let $\lambda_j\in V_H$ denote the associated nodal basis function (tent function), i.e. $\lambda_j \in V_H$ with the property $\lambda_j(z_i)=\delta_{ij}$ for all $1\le i,j \le J$.

From now on, we denote by $u_h \in V_h$ the classical finite element approximation of $u$ in the discrete (highly resolved) space $V_h$, i.e. $u_h \in V_h$ solves
\begin{align}
\label{high-resolution-fem-equation}\int_{\Omega} A \nabla u_h \cdot \nabla v_h + F( \cdot, u_h, \nabla u_h) v_h = \int_{\Omega} g v_h
\end{align}
for all $v_h \in V_h$. We assume that $V_h$ resolves the micro structure, i.e. that the error $\| u - u_h \|_{H^1(\Omega)}$ becomes sufficiently small by falling below a given tolerance.
For standard finite elements methods the error can be typically bounded by $C\cdot h^s$ with $s\ge \frac{1}{2}$. However, for regular coefficients, $C$ depends on the derivative of $A$ and $F$ with respect to the spatial variable. If $A$ and $F$ are rapidly oscillating, the derivative becomes very large and $h$ must be very small to equalize the dominance of $C$. This is only fulfilled, when $h$ resolves the micro structure (we refer 
to \cite{PS11-MMS} and \cite{P11-NHM} for some qualitative characterization of this so-called resolution condition). We are therefore dealing with pre-asymptotic effects for the standard methods. The multiscale method that we propose in the subsequent sections will be constructed to converge to $u_h$, with a linear speed in $H$ and with a generic constant that is not affected by the fine scale oscillations, i.e. in particular we do not have such pre-asymptotic effects.

\subsection{Quasi Interpolation}\label{ss:intpol}
The key tool in our construction is a linear (quasi-)interpolation operator $\Inodal: V_h\rightarrow V_H$ that is continuous and surjective. The kernel of this operator is going to be our fine space (or remainder space) $\Vf_h$. In \cite{MaPe12} a weighted Cl\'ement interpolation operator was used. In this work, we do not specify the choice. Instead, we state a set of assumptions that must be fulfilled in order to derive an optimal convergence result for the constructed multiscale method.
\begin{assumption}[Assumptions on the interpolation]
We make the following assumptions on the interpolation operator $\Inodal: V_h\rightarrow V_H$:
\begin{itemize}
\item[(A4)] $\Inodal \in L(V_h,V_H)$, i.e. $\Inodal$ is linear,
\item[(A5)] the restriction of $\Inodal$ to $V_H$ is an isomorphism, in particular there holds $(\Inodal \circ \Inodal^{-1})(v_H)=v_H$ for all $v_H \in V_H$,
\item[(A6)] there exists a generic constant $\Cint$, only depending on the shape regularity of $\Tau_H$ and $\Tau_h$, such that for all $v_h\in V_h$ and for all $T\in\triH$ there holds
\begin{equation}
\nonumber H_T^{-1}\|v_h-\Inodal(v_h)\|_{L^{2}(T)}+\|\nabla(v_h-\Inodal(v_h))\|_{L^{2}(T)}\leq \Cint \| \nabla v_h\|_{L^2(\omega_T)}
\end{equation}
with
\begin{align*}
\omega_{T} := {\bigcup \{K \in \triH| \hspace{2pt} \overline{K} \cap \overline{T} \neq \emptyset } \}.
\end{align*}
\item[(A7)] there exists a generic constant $\Cint^{\prime}$, only depending on the shape regularity of $\Tau_H$ and $\Tau_h$, such that for all $v_H \in V_H$ there exists $v_h \in V_h$ with
\begin{align*}
 \Inodal(v_h) = v_H, \quad |v_h|_{H^1(\Omega)} \le \Cint^{\prime} |v_H|_{H^1(\Omega)} \quad \mbox{and} \quad \mbox{\rm supp} \hspace{2pt} v_h \subset \mbox{\rm supp} \hspace{2pt} v_H.
\end{align*} 

\end{itemize}
\end{assumption}
%
We note that the classical nodal interpolation operator does not fulfill assumption (A6) for $d>1$ since the constant $\Cint$ is exploding for $h \rightarrow 0$. Numerical experiments confirm that such a choice leads in fact to instabilities in the later method.
One possibility is to choose $\Inodal$ as a weighted Cl\'ement interpolation operator. This construction was proposed in \cite{MaPe12}. Given $v\in H^1_0(\Omega)$, $\Inodal v := \sum_{j=1}^{J} v_j \lambda_j$ defines a (weighted) Cl\'ement interpolant with nodal values
\begin{equation}\label{e:clement}\textstyle
 v_j :=\left(\int_{\Omega} v \lambda_j\dx\right)\left/\left(\int_{\Omega} \lambda_j\dx\right)\right.
\end{equation}
for $1 \le j \le J$ (c.f. \cite{Carstensen:1999,Carstensen:Verfuerth:1999,Clement:1975}) and zero in the boundary nodes. Furthermore, there exists the desired generic constant $\Cint$ (only depending on the mesh regularity parameter and in particular independent of $H_T$) such that for all $v\in H^1_0(\Omega)$ and for all $T\in\triH$ there holds
\begin{equation}
\nonumber H_T^{-1}\|v-\Inodal v\|_{L^{2}(T)}+\|\nabla(v-\Inodal v)\|_{L^{2}(T)}\leq \Cint \| \nabla v\|_{L^2(\omega_T)}.
\end{equation}
We refer to \cite{Carstensen:1999,Carstensen:Verfuerth:1999} for a proof of this estimate. This gives us (A6). Assumptions (A4) and (A5) are obvious. The validity of (A7) was proved in \cite{MaPe12}.

Note that in certain applications a particular interpolation operator may be superior over other choices (cf. Remark 4 in \cite{MaPe12} 
).

\medskip

\subsection{Multiscale Splitting and Modified Nodal Basis}\label{ss:multiscale}

In this section, we construct a splitting of 
the high resultion finite element space $V_h$
into a low dimension multiscale space $V^{\operatorname*{ms}}$ and some high dimensional remainder space $\Vf_h$. From now on, we let $\Inodal: V_h\rightarrow V_H$ denote an interpolation operator fulfilling the properties (A4)-(A7). Recall that $V_H \subset V_h$. 
We start with defining $\Vf_h$ as the kernel of $\Inodal$ in $V_h$:
\begin{equation*}\label{e:finescale}
 \Vf_h:=\{v_h\in V_h\;\vert\;\Inodal v_h=0\}.
\end{equation*}
$\Vf_h$ represents the features in $V_h$ not captured by $V_H$. As already mentioned, it can be shown that
$(\Inodal)_{|V_H} : V_H \rightarrow V_H$ is an isomorphism (see \cite{MaPe12}). We therefore get
\begin{align}
\label{splitting-1} V_h = V_H \oplus \Vf_h, \quad \mbox{where} \enspace \underset{\in V_h}{\underbrace{v_h}} = \underset{\in V_H}{\underbrace{\Inodal^{-1}(\Inodal(v_h))}} + \underset{\in \Vf_h}{\underbrace{v_h - \Inodal^{-1}(\Inodal(v_h))}}.
\end{align}
Here, the property $(\Inodal \circ \Inodal^{-1})(v_H)=v_H$ for all $v_H \in V_H$ implies the equation $\Inodal(v_h - \Inodal^{-1}(\Inodal(v_h)))=\Inodal(v_h) - (\Inodal \circ \Inodal^{-1})( \Inodal(v_h) ) = 0$. We still need to modify the splitting of $V_h$, because $V_H$ is an inappropriate space for a multiscale approximation. We therefore look for the orthogonal complement of $\Vf_h$ in $V_h$ with respect to the inner product $\langle A \nabla \cdot, \nabla \cdot\rangle_{L^2(\Omega)}$. For this purpose, we define the orthogonal projection $P: V_h \rightarrow \Vf_h$ as follows. For a given $v_h\in V_h$, $P(v_h)\in\Vf_h$ solves
\begin{equation*}\label{e:finescaleproj}
 \langle A\nabla P(v_h),\nabla w^{\operatorname*{f}}\rangle=\langle A\nabla v_h,\nabla w^{\operatorname*{f}} \rangle\quad\text{for all }w^{\operatorname*{f}} \in \Vf_h.
\end{equation*}
Defining the multiscale space $\VmsHh$ by $\VmsHh:=(V_H- P(V_H))$, this directly leads to an orthogonal splitting:
\begin{align}
\label{splitting-2} V_h = \VmsHh \oplus \Vf_h,
\end{align}
because of
\begin{align*}
V_h = \mbox{kern}(P) \oplus \Vf_h = (V_h\hspace{-2pt}-\hspace{-2pt}P(V_h)) \oplus \Vf_h \overset{(\ref{splitting-1})}{=} (V_H\hspace{-2pt}-\hspace{-2pt}P(V_H)) \oplus \Vf_h = \VmsHh \oplus \Vf_h.
\end{align*}
Hence, any function $v_h\in V_h$ can be decomposed into $v_h=v_H^{\operatorname*{ms}} + v^{\operatorname*{f}}$ with
$v_H^{\operatorname*{ms}}=\Inodal^{-1}(\Inodal(v_h)) - P(\Inodal^{-1}(\Inodal(v_h)))$ and $v^{\operatorname*{f}}=v_h-\Inodal^{-1}(\Inodal(v_h))+ P(\Inodal^{-1}(\Inodal(v_h)))$.
Furthermore it holds $\langle A\nabla \vms,\nabla\wfine\rangle=0$ for all $\wfine \in \Vf_h$. The space $\VmsHh$ is a multiscale space of the same dimension as the coarse space $V_H$. However, note that it is only constructed on the basis of the oscillations of $A$. The oscillations of $F$ are not taken into account. We will show that $\VmsHh$ still yields the desired approximation properties. 

We now introduce a basis of $\VmsHh$. The image of the nodal basis function $\lambda_j \in V_H$ under the fine scale projection $P$ is denoted by $\phi_j^h=P(\lambda_j)\in \Vf_h$, i.e., $\phi_j^h$ satisfies the corrector problem
\begin{equation}\label{e:Tnodal}
 \langle A\nabla\phi_j^h,\nabla w\rangle=\langle A\nabla \lambda_j,\nabla w\rangle\quad\text{for all }w\in \Vf_h.
\end{equation}
A basis of $\VmsHh$ is then given by the modified nodal basis
\begin{equation}\label{e:basiscoarse}
 \{\lambda_j^{\operatorname*{ms}}:= \lambda_j-\phi_j^h\;\vert\ 1 \le j \le J\}.
\end{equation}
As we can see, solving (\ref{e:Tnodal}) involves a fine scale computation on the whole domain $\Omega$. However, since the right hand side has small support, we are able to truncate the computations. As we will see in the next section, the correctors show an exponential decay outside of the support if the coarse shape function $\lambda_j$.

$\\$
We define a (preliminary) LOD approximation without truncation.
\begin{definition}[LOD approximation without truncation]\label{definition-full-vmm-no-truncation}
The Galerkin approximation $\umsh \in \VmsHh$ of the exact solution $u$ of problem \eqref{originalproblem} is defined as the solution of
\begin{equation}
\label{equation-full-vmm-no-truncation}\langle A\nabla \umsh,\nabla v\rangle+\langle F(\umsh,\nabla\umsh),v\rangle = \langle g,v\rangle \quad\text{ for all }v\in\VmsHh.
\end{equation}
\end{definition}


\subsection{Localization}\label{ss:localization}

So far, in order to construct a suitable multiscale space, we derived a set of linear fine scale problems (\ref{e:Tnodal}) that can be solved in parallel. Still, as already mentioned in the previous section, these corrector problems are fine scale equations formulated on the whole domain $\Omega$ which makes them almost as expensive as the original problem. However, in \cite{MaPe12} it was shown that the correction $\phi_j^h$ decays with exponential speed outside of the support of the coarse basis function $\lambda_j$. We specify this feature as follows. Let $k\in\mathbb{N}_{>0}$. We define nodal patches $\omega_{j,k}$ of $k$ coarse grid layers centered around the node $z_j\in\mathcal{N}_H$ by
\begin{equation}\label{e:omega}
 \begin{aligned}
 \omega_{j,1}&:=\support \lambda_j=\cup\left\{T\in\triH\;|\;z_j\in \overline{T}\right\},\\
 \omega_{j,k}&:=\cup\left\{T\in\triH\;|\;\overline{T}\cap {\overline\omega}_{j,{k-1}}\neq\emptyset\right\} \quad \mbox{for} \enspace k\ge 2.
\end{aligned}
\end{equation}
These are the truncated computational domains for the corrector problems (\ref{e:Tnodal}). The fast decay is summarized by the following lemma:
\begin{lemma}[Decay of the local correctors \cite{MaPe12}]\label{l:decay}
Let assumptions (A1) and (A4)-(A7) be fulfilled. Then, for all nodes $z_j\in\mathcal{N}_H$ and for all $k\in\mathbb{N}_{>0}$ there holds the following estimate for the correctors $\phi_j^h$:
$$\|A^{1/2}\nabla \phi_j^h\|_{L^2(\Omega\setminus\omega_{j,k})}\lesssim e^{-(\alpha/\beta)^{1/2} k}\|A^{1/2}\nabla \phi_j^h\|_{L^2(\Omega)}.$$ Remind the definition of '$\lesssim$' at the end of Section \ref{s:setting}.
\end{lemma}
This fast decay motivates an approximation of $\phi_j^h$ on the truncated nodal patches $\omega_{j,k}$. 
We therefore define localized fine scale spaces by intersecting $\Vf_h$ with those functions that vanish outside the patch $\omega_{j,k}$, i.e.
\begin{align*}
\Vf_h(\omega_{j,k}):=\{v \in\Vf_h \;\vert\;v_{|\Omega\setminus\omega_{j,k}}=0\}
\end{align*}
for a given node $z_j\in\mathcal{N}_H$.
The solutions $\phi_{j,k}^h\in \Vf_h(\omega_{j,k})$ of
\begin{equation}\label{e:Tnodallocal}
 \langle A\nabla \phi_{j,k}^h,\nabla w\rangle=\langle A\nabla\lambda_j,\nabla w\rangle\quad\text{for all }w\in \Vf_h(\omega_{j,k}),
\end{equation}
are approximations of $\phi_j^h$ from \eqref{e:Tnodal} with local support and therefore cheap to solve. We define localized multiscale finite element spaces by
\begin{equation}
\label{e:modeldiscrete}\VmsHkh=\operatorname*{span}\{\lambda_{j,k}^{\operatorname*{ms}}:=\lambda_j-\phi_{j,k}^h\;\vert\; 1 \le j \le J\}\subset V_h.
\end{equation}
We can now define a LOD approximation including truncation:
\begin{definition}[LOD approximation with truncation]
The Galerkin approximation $\umshk\in \VmsHkh$ of the exact solution $u$ of problem \eqref{originalproblem} is defined as the solution of
\begin{equation}
\label{equation-full-vmm}\langle A\nabla \umshk,\nabla v\rangle+\langle F(\umshk,\nabla\umshk),v\rangle = \langle g,v\rangle \quad\text{ for all }v\in\VmsHkh.
\end{equation}
\end{definition}
Note, that changing $F$ and $g$ does not change the multiscale basis $\{\lambda_{j,k}^{\operatorname*{ms}}\;\vert\; 1 \le j \le J\}$. Once $\VmsHkh$ is computed, it can be reused for various data functions $F$ and $g$. This makes the new problems cheap to solve.

\begin{remark}
Observe that we never need to solve a problem on the scale of the oscillations of $F(\cdot,\xi,\zeta)$ in the case that they are faster than the oscillations of $A(\cdot)$. However, we implicitly assume that the arising integrals can be computed exactly (or with high accuracy). Practically this implies that a sufficiently high quadrature rule must be used. So even if the fine grid is not fine enough to resolve the variations of $F$, at least the quadrature rule must be fine enough to capture the correct averaged values. From Theorem \ref{main-result} below we deduce that the influence of the oscillations of $F(\cdot,\xi,\zeta)$ remains small, as long as we have an accurate approximation of the averages on each coarse grid element. A similar observation holds for standard finite elements, where classical convergence rates can be expected as soon as the oscillations of $A$ are resolved by the fine grid (independent of the oscillations of $F$).
\end{remark}

\subsection{A-priori error estimate}\label{ss:a-priori-error}

We are now prepared to state the main result of this contribution, namely the optimal convergence of the method for the case that the local patches $\omega_{j,k}$ have a diameter of order $H \log( H^{-1})$:

\begin{theorem}
\label{main-result}
Let $u \in H^1_0(\Omega)$ denote the exact solution given by problem (\ref{originalproblem}), let $u_h \in V_h$ denote the corresponding finite element approximation in the Lagrange space with a highly resolved computational grid (i.e. the solution of (\ref{high-resolution-fem-equation})) and let $\umshk \in \VmsHkh$ be the solution of our proposed multiscale method with truncation (i.e. the solution of (\ref{equation-full-vmm})). If assumptions (A1)-(A7) are fulfilled and if the diameter of the local patches $\omega_{j,k}$ fulfills O$(\mbox{\rm diam}(\omega_{j,k}))\gtrsim H \log(\|H^{-1}\|_{L^\infty(\Omega)})$ ($1\le j \le J$), then the following a-priori error estimate holds true:
\begin{align*}
\| u - \umshk \|_{H^1(\Omega)} \le C\left( L_1,L_2,\alpha,c_0 \right) \left( \| H \|_{L^{\infty}(\Omega)} + \|u-u_h\|_{H^1(\Omega)} \right).
\end{align*}
Here, $C$ denotes a generic constant with the property that it becomes large when $\frac{L_1+L_2}{\alpha}$ becomes large or when $c_0^{-1}$ becomes large. Still, $C$ is independent of the oscillations of $A$ and $F$ and only depends on the shape regularity of $\Tau_H$ and $\Tau_h$ and ratio of the magnitudes of $F$ and $A$. A suitable choice of $m\in \mathbb{N}$ with $k=m\cdot\log(\|H^{-1}\|_{L^\infty(\Omega)})$ depends on the contrast $\frac{\beta}{\alpha}$. The larger the contrast, the bigger should be $m$.
\end{theorem}

A proof of Theorem \ref{main-result} is presented in the subsequent section. In particular, the result is a conclusion from Theorem \ref{corollary-semi-linear-problem-local} which is stated in Section \ref{s:error} below. In Theorem \ref{corollary-semi-linear-problem-local} we also give details on the generic constant $C$.
We will see that it essentially depends on $\frac{(L_1+L_2)}{\alpha}$. Recall that $L_1$ and $L_2$ denote the Lipschitz constants of $F$ (c.f. (A2)) and that $\alpha$ is the smallest eigenvalue of $A$. This shows the significance of assuming that the problem is not dominated by the lower order term. For instance, consider the scenario of a pollutant being transported by groundwater flow. In this case, $A$ describes the hydraulic conductivity which changes its properties on a scale of size $\epsilon$. On the other hand, $F$ describes the gravity driven flow that is scaled with the so called P\'eclet number. However, in the described scenario the P\'eclet number is of order $\epsilon^{-1}$ (c.f. Bourlioux and Majda \cite{Bourlioux:Majda:2000}) implying that O$(L_1)=\epsilon^{-1}$. So the generic constant $C$ is of order $\epsilon^{-1}$. This means that we need $H<\epsilon$, i.e. we still need to resolve the micro structure with the coarse grid $\mathcal{T}_H$ producing the same costs as the original problem. On the contrary, if $H \gg \epsilon$ the estimate 
stated in Theorem \ref{main-result} is of no value, because the right hand side remains large.

\section{Error Analysis}\label{s:error}

This section is devoted to the proof of Theorem \ref{main-result}. In particular, we state a detailed version of the result (see Theorem \ref{corollary-semi-linear-problem-local} below), where we specify the occurring constants. The proof is splitted into several lemmata. We start with an a-priori error estimate for a LOD approximation without truncation:

\begin{lemma}
\label{lemma-semi-linear-problem-no-truncation}
Let $u_h \in V_h$ denote the highly resolved finite element approximation defined via equation (\ref{high-resolution-fem-equation}) and let $\umsh \in \VmsHh$ denote the LOD approximation given by equation (\ref{equation-full-vmm-no-truncation}). Under assumptions (A1)-(A7), the following a-priori error estimate holds true:
\begin{eqnarray*}
\lefteqn{|u_h-u_H^{\operatorname*{ms}}|_{H^1(\Omega)}}\\
&\lesssim& \tilde{C}_0 \big( \|H g\|_{L^2(\Omega)} \hspace{-2pt} + \hspace{-2pt} \|H\|_{L^{\infty}(\Omega)} C_p\frac{L_1 C_p + L_2}{c_0 } \|g\|_{L^2(\Omega)}\big),
\end{eqnarray*}
where
\begin{align*}
 \tilde{C}_0 :=  \left(\frac{\beta+\|H\|_{L^{\infty}(\Omega)}(L_{1} C_p \hspace{-2pt} + \hspace{-2pt} L_{2})}{c_0 \cdot \alpha}\right).
\end{align*}
\end{lemma}
\begin{proof}
Due to (\ref{splitting-2}), we know that there exist $\tildeumsh \in \VmsHh$ and $\tilde{u}_h^{\operatorname*{f}} \in V_h^{\operatorname*{f}}$, such that
\begin{align*}
 u_h = \tildeumsh + \tilde{u}_h^{\operatorname*{f}}.
\end{align*}
We use the Galerkin orthogonality obtained from the equations (\ref{high-resolution-fem-equation}) and (\ref{equation-full-vmm-no-truncation}) to conclude for all $v\in\VmsHh$:
\begin{align}
\label{galerkin-orthogonality}\langle A\nabla(u_h-\umsh),\nabla v\rangle+
\langle F(u_h,\nabla u_h),v\rangle-\langle F(\umsh,\nabla \umsh),v\rangle=0.
\end{align}
In particular $v=\umsh-\tildeumsh \in\VmsHh$ is an admissible test function in (\ref{galerkin-orthogonality}). Together with $\Inodal(\tilde{u}_h^{\operatorname*{f}})=0$, this yields:
\begin{eqnarray*}
\lefteqn{c_0 |u_h -\umsh |^2_{H^1(\Omega)}}\\
&\overset{(\ref{strong-monotonicity})}{\le}& \langle A\nabla(u_h- \umsh),\nabla (u_h -\umsh )\rangle \\
&\enspace& \quad + \langle F(u_h,\nabla u_h)-F( \umsh ,\nabla \umsh ),u_h -\umsh \rangle\\
&\overset{(\ref{galerkin-orthogonality})}{=}& \langle A\nabla(u_h- \umsh),\nabla (u_h - \tildeumsh )\rangle \\
&\enspace& \quad + \langle F(u_h,\nabla u_h)-F( \umsh ,\nabla \umsh ),u_h - \tildeumsh \rangle\\
&=& \langle A\nabla(u_h - \umsh ),\nabla \tilde{u}_h^{\operatorname*{f}}\rangle +
\langle F(u_h,\nabla u_h)-F(\umsh,\nabla u_h), \tilde{u}_h^{\operatorname*{f}} - \Inodal(\tilde{u}_h^{\operatorname*{f}})\rangle \\
&\enspace& \quad + \langle F(\umsh,\nabla u_h)-F(\umsh,\nabla \umsh),\tilde{u}_h^{\operatorname*{f}} - \Inodal(\tilde{u}_h^{\operatorname*{f}})\rangle\\
&\lesssim& \beta | u_h-\umsh |_{H^1(\Omega)} |
\tilde{u}_h^{\operatorname*{f}}|_{H^1(\Omega)} \\
&\enspace& \quad + \|H\|_{L^{\infty}(\Omega)} (L_{1}\| u_h-\umsh \|_{L^2(\Omega)} + L_2 |  u_h-\umsh |_{H^1(\Omega)}) |
\tilde{u}_h^{\operatorname*{f}}|_{H^1(\Omega)}\\
&\lesssim& (\beta+\|H\|_{L^{\infty}(\Omega)}(L_{1} C_p +L_{2})) \hspace{2pt} \cdot \hspace{2pt} |  u_h-\umsh |_{H^1(\Omega)}  \hspace{2pt} \cdot \hspace{2pt} |\tilde{u}_h^{\operatorname*{f}}|_{H^1(\Omega)}.
\end{eqnarray*}
With $\langle A \nabla \tildeumsh, \nabla \tilde{u}_h^{\operatorname*{f}}\rangle =0$ and with $\Inodal(v_{\operatorname*{f}})=0$ for all $v_{\operatorname*{f}} \in \Vf$ we get
\begin{eqnarray*}
\lefteqn{\alpha |\tilde{u}_h^{\operatorname*{f}}|_{H^1(\Omega)}^2\le\langle A\nabla \tilde{u}_h^{\operatorname*{f}},\nabla \tilde{u}_h^{\operatorname*{f}}\rangle}\\
&=& \langle A\nabla u_h,\nabla \tilde{u}_h^{\operatorname*{f}}\rangle = \langle g,\tilde{u}_h^{\operatorname*{f}}\rangle -\langle F(u_h,\nabla u_h),\tilde{u}_h^{\operatorname*{f}}\rangle\\
&=&\langle g ,\tilde{u}_h^{\operatorname*{f}} - \Inodal(\tilde{u}_h^{\operatorname*{f}})\rangle -\langle F(u_h,\nabla u_h),\tilde{u}_h^{\operatorname*{f}} - \Inodal( \tilde{u}_h^{\operatorname*{f}} )\rangle\\
&\overset{(\ref{bound_F_u_grad_u})}{\lesssim}& \big( \|H g\|_{L^2(\Omega)}+\|H\|_{L^{\infty}(\Omega)} C_p\frac{L_1 C_p + L_2}{c_0 } \|g\|_{L^2(\Omega)}\big) \hspace{2pt} \cdot \hspace{2pt} |\tilde{u}_h^{\operatorname*{f}}|_{H^1(\Omega)}.
\end{eqnarray*}
The theorem follows by combing the results.
\end{proof}


The next lemma is a conclusion from the previous one:
\begin{lemma}
\label{lemma-semi-linear-problem-truncation}
Let $u_h \in V_h$ denote the fine scale approximation obtained from equation (\ref{high-resolution-fem-equation}) and let $\umshk \in \VmsHkh$ denote the solution of problem (\ref{equation-full-vmm}) (fully discrete LOD with truncation). If the assumptions (A1)-(A7) hold true we obtain the following estimate:
\begin{eqnarray*}
\lefteqn{|u_h -\umshk |_{H^1(\Omega)}}\\
&\lesssim& \tilde{C}_2 \|g\|_{L^2(\Omega)} \|H\|_{L^{\infty}(\Omega)} + \tilde{C}_3 \min_{\vmshk\in\VmsHkh}\| A^{\frac{1}{2}} \nabla (\umsh - \vmshk) \|_{L^2(\Omega)},
\end{eqnarray*}
where
\begin{align*}
 \tilde{C}_1 &:= (\beta + ( L_1 C_p + L_2)C_p) \cdot \left(\frac{\beta+\|H\|_{L^{\infty}(\Omega)}(L_{1} C_p \hspace{-2pt} + \hspace{-2pt} L_{2})}{c_0^2 \cdot \alpha}\right), \\
 \tilde{C}_2 &:= \tilde{C}_1+ \tilde{C}_1 \cdot C_p\frac{L_1 C_p + L_2}{c_0 },\\
 \tilde{C}_3 &:= \frac{1 + \alpha^{-\frac{1}{2}}( L_1 C_p + L_2)C_p}{c_0}.
\end{align*}
\end{lemma}
\begin{proof}
Let $\vmshk\in\VmsHkh$ denote an arbitrary element. Using the Galerkin orthogonality obtained from (\ref{high-resolution-fem-equation}) and (\ref{equation-full-vmm}), we start in the same way as in the proof of Lemma \ref{lemma-semi-linear-problem-no-truncation} to get:
\begin{eqnarray*}
\lefteqn{c_0 |u_h -\umshk |^2_{H^1(\Omega)}}\\
&\overset{(\ref{strong-monotonicity})}{\le}& \langle A\nabla(u_h- \umshk),\nabla (u_h -\umshk )\rangle \\
&\enspace& \quad + \langle F(u_h,\nabla u_h)-F( \umshk ,\nabla \umshk ),u_h -\umshk \rangle\\
&\overset{(\ref{galerkin-orthogonality})}{=}& \langle A\nabla(u_h- \umshk),\nabla (u_h - \umsh) + \nabla (\umsh - \vmshk )\rangle \\
&\enspace& \quad + \langle F(u_h,\nabla u_h)-F( \umshk ,\nabla \umshk ),  (u_h - \umsh) + (\umsh - \vmshk ) \rangle\\
&\le& (\beta + ( L_1 C_p + L_2)C_p) |u_h- \umshk|_{H^1(\Omega)} \hspace{2pt} | u_h - \umsh |_{H^1(\Omega)} \\
&\enspace& + (1 + \alpha^{-\frac{1}{2}}( L_1 C_p + L_2)C_p) |u_h- \umshk|_{H^1(\Omega)} \hspace{2pt} \| A^{\frac{1}{2}} \nabla (\umsh - \vmshk) \|_{L^2(\Omega)}.
\end{eqnarray*} 
Dividing by $|u_h -\umshk |_{H^1(\Omega)}$ and estimating $| u_h - \umsh |_{H^1(\Omega)}$ with Lemma \ref{lemma-semi-linear-problem-no-truncation} yields the result.
\end{proof}
Combining the results of Lemma \ref{l:decay} and Lemma \ref{lemma-semi-linear-problem-truncation} yields the main result of this contribution:
\begin{theorem}
\label{corollary-semi-linear-problem-local}
Let $u_h \in V_h$ be solution of (\ref{high-resolution-fem-equation}) and let $\umshk \in \VmsHkh$ be the solution of (\ref{equation-full-vmm}). If the assumptions (A1)-(A7) hold true and if the number of layers $k$ fulfills $k\gtrsim\log(\|H^{-1}\|_{L^\infty(\Omega)})$, then it holds
\begin{align*}
|u_h -\umshk |_{H^1(\Omega)} \lesssim\tilde{C} \|H\|_{L^{\infty}(\Omega)} \|g\|_{L^2(\Omega)},
\end{align*}
where 
\begin{align*}
 \tilde{C} &:= \tilde{C}_2 + C_p \frac{\beta}{c_0} \tilde{C}_3
\end{align*}
and with $\tilde{C}_2$ and $\tilde{C}_3$ as in Lemma \ref{lemma-semi-linear-problem-truncation}.
\end{theorem}
\begin{proof}
We define $\wmshk\in\VmsHk$ by
\begin{align*}
\wmshk :=\sum_{j=1}^J \umsh(z_j) \lambda_{j,k}^{\operatorname*{ms}} = \sum_{j=1}^J \umsh(z_j) (\lambda_j-\phi_{j,k}^h)
\end{align*}
where $\umsh(z_j)$, $j=1,2,\ldots,J$, are the coefficients in the basis representation of $\umsh$ from Definition~\ref{definition-full-vmm-no-truncation}.
Hence, 
\begin{equation}\label{e:est1}
\begin{aligned}
\lefteqn{\min_{\vmshk\in\VmsHkh}\| A^{\frac{1}{2}} \nabla (\umsh - \vmshk) \|_{L^2(\Omega)}^2}\\
&\leq \| A^{\frac{1}{2}} \nabla (\umsh - \wmshk) \|_{L^2(\Omega)}^2\\
&\lesssim \sum_{j = 1}^J k^d \umsh(z_j)^2 \|A^{1/2}\nabla (\phi_{j}^h-\phi_{j,k}^h)\|^2_{L^2(\Omega)}.
\end{aligned}
\end{equation}
For details on the last step, we refer to Lemma 18 in  \cite{MaPe12}. Due to the Galerkin orthogonality for the corrector problems 
we get 
\begin{equation*}
 \|A^{1/2}\nabla (\phi_{j}^h-\phi_{j,k}^h)\|^2_{L^2(\Omega)}\lesssim \|A^{1/2}\nabla \phi_{j}^h\|^2_{L^2(\Omega\setminus\omega_{j,k-1})}.
\end{equation*}
We refer to the first part of the proof of Lemma 8 in \cite{MaPe12} for details.
The application of Lemma~\ref{l:decay}, \eqref{e:Tnodal} and some inverse inequality yield
\begin{eqnarray*}
 \|A^{1/2}\nabla (\phi_{j}^h-\phi_{j,k}^h)\|^2_{L^2(\Omega)}&\lesssim & e^{-2(\alpha/\beta)^{1/2} k}\|A^{1/2}\nabla \phi_j^h\|_{L^2(\Omega)}^2\\
&\leq & e^{-2(\alpha/\beta)^{1/2} k}\|A^{1/2}\nabla\lambda_j^h\|_{L^2(\Omega)}^2\\
&\leq & \beta e^{-2(\alpha/\beta)^{1/2} k} \Vert H^{-1} \Vert_\infty^2 \|\lambda_j^h\|_{L^2(\Omega)}^2
\end{eqnarray*}
By choosing $k = m \cdot \log(\|H^{-1}\|_{L^\infty(\Omega)})$ with $m \in \mathbb{N}$, we can achieve an arbitrary fast polynomial convergence of this term in $H$ (this will also cancel the $k^d$ term). 
However, we bound this by a linear convergence since this is fastest rate that we can obtain for the whole error. Finally, the combination of this estimate and \eqref{e:est1} plus 
$\sum_{j = 1}^J \umsh(z_j)^2 \|\lambda_{j}^h\|^2_{L^2(\Omega)}\lesssim\|\Inodal \umsh\|^2_{L^2(\Omega)}\lesssim \|\nabla \umsh\|_{L^2(\Omega)}^2 \le C_p^2 c_0^{-2}\|g\|_{L^2(\Omega)}^2 $ yields the assertion.
\end{proof}

\section{The Multiscale Newton scheme}\label{s:newtonms}

In this section we discuss a solution algorithm for handling the nonlinear multiscale problem (\ref{equation-full-vmm}). For this purpose, we consider a damped Newton's method in the multiscale space $\VmsHkh$. Remind the considered problem: we are looking for $u\in H_{0}^{1}( \Omega )$ with
\begin{align*}
\langle B(u), v \rangle_{H^{-1},H^1_0}  =\langle g, v\rangle\quad\text{for all } v\in H^1_0(\Omega),
\end{align*}
where we introduced the notation
\begin{align*}
\langle B(v), w \rangle_{H^{-1},H^1_0} := \langle A\nabla v,\nabla w\rangle+\langle F(\cdot,v,\nabla v),w\rangle.
\end{align*}
Here, $B : H^1_0(\Omega) \rightarrow H^{-1}(\Omega)$
is a hemicontinuous and strongly monotone operator due to assumption (A3). As already mentioned, under these assumptions, the Browder-Minty theorem yields a unique solution of the above problem.
However, we will need an additional assumption on $F$ to guarantee that the Newton scheme converges:
\begin{assumption}
\label{assumption-d-F-lipschitz}
Let $DF(x,\cdot,\cdot)$ denote the Jacobian matrix of $F(x,\cdot,\cdot)$.
\begin{enumerate}
 \item[(A8)] We assume that there exists some constant $L_D\ge0$ so that for almost every $x$ in $\Omega$ and for all $(\xi_1,\zeta_1) \in \R \times \R^d$ and $(\xi_2,\zeta_2) \in \R \times \R^d$
\begin{align*}
 | DF(x,\xi_1,\zeta_1) - DF(x,\xi_2,\zeta_2)| \le L_D | (\xi_1,\zeta_1) - (\xi_2,\zeta_2)|,
\end{align*}
i.e. $F(x,\cdot,\cdot)\in W^{2,\infty}(\R \times \R^d)$.
\end{enumerate}
\end{assumption}
For clarity of the presentation we will leave out several indices within this section. In particular, we make use of the following notation:
\begin{definition}
\label{simplifying-notation}For simplicity, we define
\begin{align*}
V^{\operatorname*{\operatorname*{ms}}}:=\VmsHkh \quad \mbox{with basis} \enspace \lambda_{j}^{\operatorname*{ms}}:=\lambda_{j,k}^{\operatorname*{ms}}=\lambda_j-\phi_{j,k}^h \enspace \mbox{for} \enspace 1 \le j \le J.
\end{align*}
Furthermore, we denote $\ums := \umshk$. Additionally, let
\begin{align*}
 \partial_1 F(x,\xi,\zeta) := \partial_{\xi} F(x,\xi,\zeta) \quad \mbox{and} \quad \partial_2 F(x,\xi,\zeta) := \partial_{\zeta} F(x,\xi,\zeta).
\end{align*}
\end{definition}

We now describe the Newton strategy in detail. The fully discrete multiscale problem to solve reads:
\begin{eqnarray*}
\mbox{find} \enspace \ums \in \Vms: \quad \langle A\nabla \ums,\nabla \lambda_{j}^{\operatorname*{ms}} \rangle+\langle F(\cdot,\ums,\nabla \ums), \lambda_{j}^{\operatorname*{ms}} \rangle - \langle g, \lambda_{j}^{\operatorname*{ms}} \rangle = 0
\end{eqnarray*}
for all $1 \le j \le J$. Again, using Browder-Minty, $\ums$ exists and is unique. Accordingly, we get the following well posed algebraic version of the problem:
\begin{eqnarray*}
\mbox{find} \enspace \bar{\alpha} \in \R^J: \quad G(\bar{\alpha}) = 0
\end{eqnarray*}
and where $G : \R^J \rightarrow \R^J$ is given by
\begin{eqnarray}
\label{definition-G-alpha-k}\lefteqn{\left({G}(\alpha)\right)_l}\\
\nonumber&:=& \sum_{j=1}^J \alpha_j \langle A\nabla\lambda_{j}^{\operatorname*{ms}} ,\nabla\lambda_{l}^{\operatorname*{ms}} \rangle+\langle F(\cdot,\sum_{j=1}^J \alpha_j \lambda_{j}^{\operatorname*{ms}} ,\sum_{j=1}^J \alpha_j\nabla \lambda_{j}^{\operatorname*{ms}} ),\lambda_{l}^{\operatorname*{ms}} \rangle - \langle g,\lambda_{l}^{\operatorname*{ms}} \rangle.
\end{eqnarray}
We have the relation $\ums = \sum_{j=1}^J \bar{\alpha}_j \lambda_{j}^{\operatorname*{ms}}$. Before we can apply the Newton method to (\ref{definition-G-alpha-k}), we need to ensure that the iterations of the scheme are well defined. The following lemma ensures this:
\begin{lemma}
\label{lemma-newton-scheme-well-posed}Let $(X,\|\cdot\|_X)$ denote a Hilbert space with dual space $X^{\prime}$. Let furthermore $B : X \rightarrow X^{\prime}$ be a  hemicontinuous, Fr\'{e}chet differentiable and strongly monotone operator on $X$, i.e. there exists $c_0>0$ so that
\begin{align*}
\langle B(v) - B(w), v-w \rangle_X &\ge c_0 \| v - w\|_{X}^2 \quad \mbox{for all} \enspace v,w \in X \enspace \mbox{and}\\
s \mapsto \langle B(u+sv) , w \rangle_X
\end{align*}
is a continuous function on $[0,1]$ for all $u,v,w \in X$. Let $X_N$ denote a finite dimensional subspace with basis $\{\psi_1,...,\psi_N\}$ and let $b: \R^N \rightarrow V_N$ define the linear bijection with $b(\alpha):=\sum_{i=1}^N \alpha_i \psi_i$.
If $G(\alpha):=b^{-1}(B(b(\alpha)))$, then the Jacobi matrix $DG(\alpha) \in \R^{n \times n}$ has only positive eigenvalues.
\end{lemma}

\begin{proof}
Let $B^{\prime}$ denote the Fr\'{e}chet derivative of $B$, given by
\begin{align*}
B^{\prime}(u)(v) = \lim_{s \rightarrow 0} \frac{B(u+sv) - B(u)}{s} \quad \mbox{for} \enspace u,v \in X.
\end{align*}
This and the strong monotonicity yield:
\begin{align}
\nonumber \langle B^{\prime}(u)(v), v \rangle_{H^{-1},H^1_0} &= \lim_{s \rightarrow 0} \frac{(B(u+sv) - B(u))(v)}{s} \\
\label{coercivity-frechet-derivative}&= \lim_{s \rightarrow 0} \frac{1}{s^2} (B(u+sv) - B(u))(u+sv-u) \\
\nonumber &\ge \lim_{s \rightarrow 0} \frac{1}{s^2} c_0 \|sv\|^2 = c_0 \|v\|^2.
\end{align}
Next, observe that $b$ induces an inner product on $\R^N$ by $(\alpha_1,\alpha_2)_b:=\langle b(\alpha_1),b(\alpha_2)\rangle_X$. Let $\alpha:=b^{-1}(u)$ then we get
\begin{align*}
B^{\prime}(u)(\psi_i) &= \lim_{s \rightarrow 0} \frac{B(u+s\psi_i) - B(u)}{s}\\
&= \lim_{s \rightarrow 0} \frac{(b\circ b^{-1})(B (\sum_{j=1}^N (\alpha_{j} + s \delta_{ij}) \psi_j ) - (b\circ b^{-1})(B(\sum_{j=1}^N \alpha_{j} \psi_j ))}{s} \\
&= b\left( \lim_{s \rightarrow 0} \frac{ G(\alpha + s e_i) - G(\alpha)}{s} \right) \\
&= b (D_{\alpha}G(\alpha) e_i).
\end{align*}
Using this, we get for arbitrary $\xi \in \R^N$ and $v_{\xi}:=b(\xi)$:
\begin{align*}
( D_{\alpha}G(\alpha) \xi, \xi )_b &= \sum_{i,j}^N \xi_i \xi_j ( D_{\alpha}G(\alpha) e_i, e_j )_b \\
&= \sum_{i,j}^N \xi_i \xi_j ( b(D_{\alpha}G(\alpha) e_i), b(e_j) )_X \\
&= \sum_{i,j}^N \xi_i \xi_j ( B^{\prime}(u)(\psi_i), \psi_j )_X \\
&= ( B^{\prime}(u)(v_{\xi}),v_{\xi} )_X \overset{(\ref{coercivity-frechet-derivative})}{\ge} c_0 \|v_{\xi}\|_X^2 = c_0 \|\xi\|_b^2.
\end{align*}
Since all norms in $\R^N$ are equivalent we have the desired result.
\end{proof}
%

%

$\\$
Now, we can apply the damped Newton method for solving the nonlinear algebraic equation $G(\bar{\alpha}) = 0$. If $D_{\alpha}G$ denotes the Jacobian matrix of $G$, we get the following iteration scheme:
\begin{eqnarray*}
 \alpha^{(n+1)} := \alpha^{(n)} + \triangle \alpha^{(n)},
\end{eqnarray*}
where $\triangle \alpha^{(n)}$ solves
\begin{eqnarray}
\label{newton-iteration}D_{\alpha}G(\alpha^{(n)}) \triangle \alpha^{(n)} = - G(\alpha^{(n)}).
\end{eqnarray}
Here, $D_{\alpha}G$ is given by:
\begin{align*}
D_{\alpha_i}\left(G(\alpha)\right)_l &:=
\langle A\nabla \lambda_{i}^{\operatorname*{ms}},\nabla \lambda_{l}^{\operatorname*{ms}}\rangle + \langle \partial_1 F(\cdot,\sum_{j=1}^J \alpha_j \lambda_{j}^{\operatorname*{ms}},\sum_{j=1}^J \alpha_j\nabla \lambda_{j}^{\operatorname*{ms}}) \lambda_{i}^{\operatorname*{ms}},\lambda_{l}^{\operatorname*{ms}}\rangle\\
&\qquad + \langle \partial_2 F(\cdot,\sum_{j=1}^J \alpha_j \lambda_{j}^{\operatorname*{ms}},\sum_{j=1}^J \alpha_j\nabla \lambda_{j}^{\operatorname*{ms}})\cdot \nabla \lambda_{i}^{\operatorname*{ms}},\lambda_{l}^{\operatorname*{ms}}\rangle.
\end{align*}
Lemma \ref{lemma-newton-scheme-well-posed} ensures that equation (\ref{newton-iteration}) has a unique solution $\triangle \alpha^{(n)}$, i.e. that the Newton iteration is well posed. Since $G\in C^1(\R^N)$ has a nonsingular Jacobian matrix $D_{\alpha}G$ (due to Lemma \ref{lemma-newton-scheme-well-posed}) and since we have Lipschitz-continuity of $D_{\alpha}G$ (due to Assumption \ref{assumption-d-F-lipschitz}), we have the that Newton scheme converges quadratically as long as the starting value is close enough to the exact solution (c.f. \cite{Dennis:Schnabel:1996}). However, this means that we can only guarantee local convergence of the method. In order to ensure global convergence, we can use a simple damping strategy due to Armijo \cite{Armijo:1966}. Here we are looking for a damping parameter $\zeta \in (0,1]$ so that $\alpha^{(n+1)} :=  \alpha^{(n)} + \zeta \triangle \alpha^{(n)}$ with the property $|G( \alpha^{(n+1)} )| < (1-\frac{\zeta}{2}) |G(\alpha^{(n)} )|$. Under the same assumptions (i.e. (A1)-(A3) and (A8)), we get that the damped Newton scheme 
converges,
i.e. there exists a nonempty (damping) interval $[\zeta_0,\zeta_1]\subset(0,1)$, so that we obtain the above 'damping property' for any $\zeta \in [\zeta_0,\zeta_1]$. Here, $\zeta_0>0$ is independent of $\alpha^{(n)}$ and $\triangle \alpha^{(n)}$, which prevents $\zeta_1\rightarrow 0$. The existence of a damping parameter so that $|G( \alpha^{(n+1)} )| < |G(\alpha^{(n)} )|$ is an easy observation if we look at the function $h(\zeta):=| G(\alpha^{(n)} + \zeta \triangle \alpha^{(n)})|^2$ which fulfills $h(0)>0$ and $h^{\prime}(0)=-2 G(\alpha^{(n)})\cdot G(\alpha^{(n)})<0$. The existence of a uniform lower bound $\zeta_{0}>0$ was proved by Kelley \cite{Kelley:1995}, Lemma 8.2.1 and Theorem 8.2.1 therewithin. The results by Kelley require Lipschitz continuity of $D_{\alpha}G$ (assumption (A8)) and uniform boundedness of $| (D_{\alpha} G(\alpha))^{-1} |$. The latter one is fulfilled since the proof of Lemma \ref{lemma-newton-scheme-well-posed} shows that the smallest eigenvalue of $(D_{\alpha} G(\alpha))$ is equal or larger than $c_0$. This implies that the largest eigenvalue of $(D_{\alpha} G(\alpha))^{-1}$ is bounded by $c_0^{-1}$, hence $| (D_{\alpha} G(\alpha))^{-1} |$ is uniformly bounded. In summary we have globally linear convergence of the method (using damping) and locally (i.e. in an environment of the solution) even quadratic convergence using the classical Newton scheme without damping.

$\\$
With these considerations, we can state the full algorithm below. Recall that $\mathcal{N}_H$ denotes the set of interior vertices of $\triH$ and for $z_j\in\mathcal{N}_H$, $\lambda_j\in V_H$ denotes the corresponding nodal basis function.

Note that in the presented algorithm, each iteration starts with the damping parameter $\zeta_n = 1$ and we do not use damping parameters from previous iterations. The advantage is that we automatically get quadratic convergence of the Newton scheme as soon as we leave the region where damping is required. Therefore, damping is only used when really necessary.

\begin{algorithm}
\label{algorithm=newton-method-damping}

 \rule{0.99\textwidth}{.7pt} \\
Algorithm: dampedNewtonLOD( $abstol$, $reltol$, $\alpha^{(0)}$, $k$ )
 \rule{0.99\textwidth}{.7pt} \\
In parallel \ForEach{$z_j \in \mathcal{N}_H $}
 {
   compute $\phi_{j,k}^h \in \Vf_h(\omega_{j,k})$ with
\begin{align*}
 \langle A\nabla\phi_{j,k}^h,\nabla w\rangle=\langle A\nabla \lambda_j,\nabla w\rangle\quad\text{for all }w\in \Vf_h(\omega_{j,k}).
\end{align*}
 }
Set $\VmsHkh := \mbox{\rm span}\{\lambda_j-\phi_{j,k}^h\;\vert\; 1 \le j \le J\}$. Set $\lambda_{j,k}^{\operatorname*{ms}}=\lambda_j-\phi_{j,k}^h$.\\
 \rule{0.99\textwidth}{.7pt} \\
Set $\alpha^{(n)} := \alpha^{(0)}$. Set $\umshkn :=\sum_{j=1}^J \alpha_j^{(n)} \lambda_{j,k}^{\operatorname*{ms}}$. Set \vspace{-10pt}
\begin{align*}
\left(G(\alpha)\right)_i :=
\sum_{j=1}^J \alpha_j \langle A\nabla \lambda_{j,k}^{\operatorname*{ms}},\nabla \lambda_{i,k}^{\operatorname*{ms}} \rangle
+\langle F(\cdot,\sum_{j=1}^J \alpha_j \lambda_{j,k}^{\operatorname*{ms}},\sum_{j=1}^J \alpha_j\nabla \lambda_{j,k}^{\operatorname*{ms}})-g,\lambda_{i,k}^{\operatorname*{ms}}\rangle.
\end{align*}
Set $tol := | G(\alpha^{(0)}) |_2 \hspace{2pt} \cdot \hspace{2pt} reltol + abstol$. \\
 \rule{0.99\textwidth}{.7pt} \\
\While{$| G(\alpha^{(n)}) |_2 > tol$}
{
Set $\umshkn :=\sum_{j=1}^J \alpha_j^{(n)} \lambda_{j,k}^{\operatorname*{ms}}$.\\
Define the entries of the stiffness matrix $M^{(n)}$ by:
\begin{eqnarray*}
M_{il}^{(n)} &:=& \langle A\nabla \lambda_{l,k}^{\operatorname*{ms}},\nabla \lambda_{i,k}^{\operatorname*{ms}}\rangle + \langle \partial_1 F(\cdot,\umshkn,\nabla \umshkn) \lambda_{l,k}^{\operatorname*{ms}},\lambda_{i,k}^{\operatorname*{ms}}\rangle\\
&\enspace&\qquad + \langle \partial_2 F(\cdot,\umshkn,\nabla \umshkn) \cdot \nabla \lambda_{l,k}^{\operatorname*{ms}},\lambda_{i,k}^{\operatorname*{ms}}\rangle.
\end{eqnarray*}
Define the entries of the right hand side by:
\begin{eqnarray*}
 F_{i}^{(n)} := \langle g,\lambda_{i,k}^{\operatorname*{ms}}\rangle - \langle A\nabla \umshkn,\nabla \lambda_{i,k}^{\operatorname*{ms}} \rangle - \langle F(\cdot,\umshkn,\nabla \umshkn),\lambda_{i,k}^{\operatorname*{ms}}\rangle.
\end{eqnarray*}
Find $(\triangle \alpha)^{(n+1)} \in \R^J$, with
\begin{eqnarray*}
 M^{(n)} (\triangle \alpha)^{(n+1)} = F^{(n)}.
\end{eqnarray*}
Set $\zeta_n := 1$. Set $\alpha^{(n+1)} :=  \alpha^{(n)} + \zeta_n \triangle \alpha^{(n)}$.\\
\While{ $|G( \alpha^{(n+1)} )| \ge (1-\frac{\zeta_n}{2}) |G(\alpha^{(n)} )|$  }
{
 Set $\zeta_n := \frac{1}{2} \zeta_n$. Set $\alpha^{(n+1)} :=  \alpha^{(n)} + \zeta_n \triangle \alpha^{(n)}$.
}

Set $\alpha^{(n)} := \alpha^{(n+1)}$. Set $tol := | G(\alpha^{(n)}) |_2 \hspace{2pt} \cdot \hspace{2pt} reltol + abstol$.
}
 \rule{0.99\textwidth}{.7pt} \\
Set $\umshkn :=\sum_{j=1}^J \alpha_j^{(n)} \lambda_{j,k}^{\operatorname*{ms}}$.
 \rule{0.99\textwidth}{.7pt} \\
\end{algorithm}

\begin{proposition}
We use the notation stated in Definition \ref{simplifying-notation}. Let $u \in H^1_0(\Omega)$ denote the solution of (\ref{originalproblem}), let $u_h \in V_h$ denote the solution of (\ref{high-resolution-fem-equation}) and let $\ums \in \Vms$ denote the solution of (\ref{equation-full-vmm})). Furthermore, we let $\umsn:=\umshkn$ define the $n$'th iterate from the damped Newton Variational Multiscale Method stated in the algorithm. Under assumptions (A1)-(A8), the Newton step (\ref{newton-iteration}) is well posed, yields an unique solution and $\umsn$ converges at least linearly to $\ums$. If furthermore $k\gtrsim\log(\|H^{-1}\|_{L^\infty(\Omega)})$, the following a-priori error estimates hold true:
\begin{eqnarray*}
\| u - \ums \|_{H^1(\Omega)} \le C \left( \| H \|_{L^{\infty}(\Omega)} + \|u-u_h\|_{H^1(\Omega)}\right)
\end{eqnarray*}
where $C$ is a generic constant with the property $C=$O$(1)$ (see Theorem \ref{main-result} and \ref{corollary-semi-linear-problem-local} for details) and
\begin{eqnarray*}
\| \ums - \umsn \|_{H^1(\Omega)} \le  L_n(H) \| \ums - u^{\operatorname*{ms},(n-1)} \|_{H^1(\Omega)}.
\end{eqnarray*}
Here, we have $L_n(H)<1$.

If $u^{\operatorname*{ms},(n-1)}$ is sufficiently close to $\ums$, we even get quadratic convergence of the Newton scheme, i.e.:
\begin{eqnarray*}
\| \ums - \umsn \|_{H^1(\Omega)} \le  L_n(H) \| \ums - u^{\operatorname*{ms},(n-1)} \|_{H^1(\Omega)}^2.
\end{eqnarray*}
with
\begin{eqnarray*}
L_n(H) \le \frac{\|(D_{\alpha}G)^{-1}\|_{L^{\infty}(\R^N)}}{L},
\end{eqnarray*}
where $L$ denotes the Lipschitz-constant of $D_{\alpha}G$. As indicated, $L_n(H)$ typically depends on the mesh size. However, in some cases of semi-linear problems, it is possible to bound $L_n(H)$ independent of the triangulation (c.f. \cite{Karatson:2012}). In particular, if $F(x,u,\nabla u)=F(x,u)$ (i.e. no dependency on $\nabla u$) we get that $L_n(H)=L_n$ independent of the underlying mesh. The proof can be obtained analogously to the proof of Proposition 4.1 in \cite{Karatson:2012}. The proof fails for general $F(x,u,\nabla u)$.
\end{proposition}

\begin{remark}
Note that the proposed method only requires the computation of the multiscale basis $\{ \lambda_{j}^{\operatorname*{ms}}| \hspace{2pt} 1 \le j \le  J\}$ once at the beginning. For each iteration step of the damped Newton scheme, (\ref{newton-iteration}) is a low dimensional linear problem that can reuse the initially computed multiscale basis. If the multiscale basis was computed using the nonlinear term $F$, local corrector problems would have to be solved for each Newton step newly, making the whole procedure significantly more expensive. We also note that assemblation of the tangent matrix $M^{(n)}$ and the residual $F^{(n)}$ still requires a quadrature rule that captures the fine scale features. Depending on the type of the nonlinearity this might have to be done newly for each iteration step, making the quadrature rule a significant part of each Newton step.
\end{remark}

\section{Numerical Experiment}
\label{s:numericalexperiments}

As mentioned in the introduction, Richards-type equations can be an application of our LOD-Newton framework. In general, the stationary Richards equation cannot necessarily be described by a monotone operator, however depending on the chosen model and the considered hydrological effects (including hysteresis, root uptake, friction, reaction fronts etc.) monotone operators can arise in certain applications. One explicit example is the (regularized) time-discretized Kirchhoff transformed Richards equation regarded in \cite{Berninger:Kornhuber:Sander:2011}. For the case that there is no Signorini boundary condition prescribed, the problem that has to be solved for each time step corresponds to a nonlinear elliptic monotone problem (on the full space) that also fulfills the required assumption of Lipschitz-continuity.

Let us now consider the stationary Kirchhoff-transformed Richards equation
\begin{align}
\label{ri-eq-model-div-form}\nabla \cdot ( K \nabla u ) - \nabla \cdot ( K \hspace{2pt} kr(M(u)) \vec{g}) = f,
\end{align}
where $u$ denotes the generalized pressure, $K$ the hydraulic conductivity and $kr$ the relative permeability depending on the saturation. $kr$ is a monotone increasing function with values between $0$ and $1$ (typically bounded away from $0$ to avoid degeneracy). If we have already full saturation, water cannot be conducted anymore, if the soil is completely dry (saturation is zero), water can be perfectly conducted. Formulas for $kr$ were e.g. provided by Burdine \cite{Burdine:1953} and Mualem \cite{Mualem:1976}. In applications the variations of the hydraulic conductivity $K$ are assumed to be constant (or at least slow) in gravity direction $\vec{g}=(0,0,\vec{g}_z)$, where $\vec{g}_z$ denotes the gravity factor of $9.81 m / s^2$. Soil probes are often only taken once in vertical direction, but a lot of samples are required to describe the variations of conductivity in horizontal direction. As a reduction of complexity one can often assume that $\nabla \cdot ( K \vec{g}) =  \partial_z( K_{zz}  \vec{g}_z ) = 0$ to consider the reduced equation
\begin{align}
\label{ri-eq-model}\nabla \cdot ( K \nabla u ) -  \hspace{2pt}  (kr \circ M)^{\prime}(u) \hspace{2pt} (K \vec{g}) \cdot \nabla u = f.
\end{align}
Here we have $M(u):=\theta \circ \kappa^{-1}$, where $\theta$ denotes the saturation (depending on the pressure) and $\kappa^{-1}$ the inverse of the Kirchhoff transformation $\kappa( p ):=\int_0^p kr(\theta(q)) \hspace{2pt} dq$. The saturation $\theta$ can be obtained by the capillary pressure relation (soil-water retention curves). Various explicit formulas for $\theta$ are available, see e.g. Van Genuchten \cite{vanGenuchten:1980}, Brooks-Corey \cite{Brooks:Corey:1964} or the Gardner model \cite{Gardner:1958}. Depending on the chosen model $(kr \circ M)^{\prime}$ might not be a Lipschitz continuous function, still regularization is possible. In the following numerical experiment, we consider a test problem that has the structure derived from a regularized Burdine-Brooks-Corey model. The corresponding explicit formulas for $(kr \circ M)$ are taken from \cite{Berninger:2007}. Contrary to the model (\ref{ri-eq-model}), we use a nonlinear advection term that is faster oscillating than the diffusion term. The reason is that we want to emphasize our claim, that the oscillations of the nonlinearity $F$ do in fact not influence the convergence. Before stating the test problem related to (\ref{ri-eq-model}), let us note that the method and the analytical results of this paper directly transfer to equations in divergence form like (\ref{ri-eq-model-div-form}), i.e. the gradient in the weak formulation can be on the test function, as long as $F(x,u)$ does not dependent on the gradient $\nabla u$.

We consider a the following nonlinear advection-diffusion problem: let $\Omega := ]0,1[^2$ and $\epsilon:=0.05$. Find $u^{\epsilon}$ with
\begin{align*}
- \nabla \cdot \left( A^{\epsilon}(x) \nabla u^{\epsilon}(x) \right) +  \frac{1}{2} F^{\epsilon}(x,u^{\epsilon}) \partial_{x_2} u^{\epsilon}(x)
&= -\frac{3}{10} \quad \mbox{in} \enspace \Omega \\
u^{\epsilon}(x) &= 0 \hspace{28pt} \mbox{on} \enspace \partial \Omega,
\end{align*}
where $A^{\epsilon}$ is given by
\begin{eqnarray*}
A^{\epsilon}(x_1,x_2):= \frac{1}{8 \pi^2} \left(\begin{matrix}
                         2(2 + \mbox{\rm cos}( 2 \pi \frac{x_1}{\epsilon} ))^{-1}  & 0 \\
                         0 & 1 + \frac{1}{2}\mbox{\rm cos}( 2 \pi \frac{x_1}{\epsilon} )
                        \end{matrix}\right)
\end{eqnarray*}
and
\begin{align*}
F^{\epsilon}(x,u) :=\frac{1}{8 \pi^2} \left( 2 + \cos( 2 \pi \frac{x_1}{\epsilon^{\frac{3}{2}}} ) \right)
\begin{cases} \sqrt{ \frac{u}{2}  + \frac{3}{2}} \enspace &\mbox{for} \enspace -3 \le u \le -\frac{5}{4}\\
p(u)\enspace &\mbox{for} \enspace -\frac{5}{4} \le u \le -1\\
0 \enspace &\mbox{for} \enspace u \ge -1
\end{cases},
\end{align*}
where $p(u) = a u^3 + b u^2 + c u + d$ is such that $F^{\epsilon}(x,\cdot)\in C^1(-3,\infty)$ for all $x \in \Omega$. The (unknown) exact solution of this problem takes values between $0$ and $-1.75$.

The numerical experiments presented in this section were performed with a little different implementation of the localization strategy than the one described in Section \ref{ss:localization}. We used the localized basis functions proposed in \cite{HP12}, which have the completely same analytical properties than (\ref{e:Tnodallocal})-(\ref{e:modeldiscrete}), with the only difference that they are computed with respect to unit vectors instead of gradients of basis functions in order to slightly stabilize the computations. The tolerance $tol$ in the Newton algorithm is set to $10^{-10}$. We keep the resolution of the (uniformly refined) fine grid fixed with $h=2^{-6}<\epsilon$. The computations were made for four different coarse grid resolutions $H=2^{-2}$,. \hspace{0pt}.., $2^{-5}$. For given $H$, we guess the number of coarse layers by $\log(H^{-1})$. By $\log$ we mean the logarithm to the basis $e$. For $H=2^{-k}$, $k=2,. . .5$ we obtain $\log(4)\approx1.386$, $\log(8)\approx2.08$, $\log(16)\approx 2.77$ and $\log(32)\approx3.47$. Optimistically rounding we set the number of coarse layers to $1.5$ for $H=2^{-2}$, $2$ for $H=2^{-3}$, $2.5$ for $H=2^{-4}$ and $3$ for $H=2^{-5}$. The corresponding results are depicted in Table \ref{table-layers-results}.
\begin{table}[t]
\caption{Results for fine grid with $\epsilon > h=2^{-6}\approx 0.016 > \epsilon^{\frac{3}{2}}$ which resolves the oscillations of the linear term, but not the oscillations of the nonlinear term. We observe an average EOC of 2.21 for the $L^2$-error and an average EOC of 1.09 for the $H^1$-error.}
\label{table-layers-results}
\begin{center}
\begin{tabular}{|c|c|c|c|c|}
\hline $H$      & coarse layers& fine layers & $\|\umshk - u_h\|_{L^2(\Omega)}$ & $\|\umshk - u_h\|_{H^1(\Omega)}$ \\
\hline
\hline $2^{-2}$ & 1.5 & 24 & 0.0299 & 0.5331 \\
\hline $2^{-3}$ & 2    & 16 & 0.0075 & 0.2825  \\
\hline $2^{-4}$ & 2.5 & 12 & 0.0017 & 0.1213 \\
\hline $2^{-5}$ & 3    & 8   & 0.0003 & 0.0550 \\
\hline
\end{tabular}\end{center}
\end{table}
We observe that the proportionality coefficient in the choice of the diameter of the patches O$(\mbox{\rm diam}(\omega_{j,k}))\sim H \log(\|H^{-1}\|_{L^\infty(\Omega)})$ can chosen to be on $1$ without suffering from pre-asymptotic effects. In fact, we obtain an experimental order of convergence (EOC) of $2.21$ for the $L^2$-error and an EOC of 1.09 for the $H^1$-error. The patches remain small and computational demand for solving the local problems remains very small. For further numerical studies of the method and the choice of patch sizes in the linear case, we refer to \cite{MaPe12}.

$\\$
{\bf{Acknowledgements.}} We would like to thank the anonymous reviewers for their helpful suggestions and their accurate feedback that helped us to improve the paper.

\end{document}